\documentclass[11pt,epsfig,amsfonts]{amsart}
\usepackage{amssymb, amsmath, amsthm}
\usepackage[colorlinks=true,linkcolor=blue,citecolor=red]{hyperref}
\usepackage{color}
\usepackage{epsfig}
\usepackage[mathscr]{euscript}
\usepackage{amsmath}
\usepackage{amssymb}
\usepackage{amscd}
\usepackage{graphicx}
\numberwithin{equation}{section}
\newtheorem{lemma}{Lemma}[section]
   \newtheorem{theorem}{Theorem}
   \newtheorem{re}{Remark}[section]
   \newtheorem{prop}{Proposition}[section]
\newcommand{\al}{\alpha}

\newcommand{\e}{\epsilon}

   \newtheorem{coro}{Corollary}[section]

\usepackage{epsfig}

\numberwithin{equation}{section}
\numberwithin{theorem}{section}
\numberwithin{prop}{section}
\numberwithin{lemma}{section}
\numberwithin{re}{section}
\numberwithin{coro}{section}

\newcommand{\R}{\mathbb{R}}

\newcommand{\cat}{\mathcal{T}}
\newcommand{\caa}{\mathcal{A}}



\thanks{ Email: \dag liji@hust.edu.cn, \ddag yliu@uta.edu, \S {wuq@ohio.edu}. }

\begin{document}

\title[Stability of DP solitary waves]{Orbital Stability of smooth solitary waves for the Degasperis-Procesi Equation}
\today

\maketitle

\centerline{\scshape Ji Li$^{\,\dag}$, Yue Liu$^{\,\ddag}$ and Qiliang Wu$^{\,\S,*}$}
\medskip
{\footnotesize
\centerline{\dag. School of Mathematics and Statistics, Huazhong University of Science and Technology}
   \centerline{Wuhan, Hubei 430074, China}
 \centerline{\ddag. Department of Mathematics, University of Texas at Arlington}
  \centerline{Arlington, TX 76019-0408, USA}
\centerline{\S. Department of Mathematics, Ohio University}
  \centerline{Athens, OH, 45701, USA}
  \centerline{$*$ the corresponding author}
   }

\begin{abstract}
The  Degasperis-Procesi equation is the integrable Camassa-Holm-type model which is an asymptotic approximation for the unidirectional propagation of shallow water waves.
This work establishes the orbital stability of localized smooth solitary waves to the   Desgasperis-Procesi (DP) equation on the real line. 
The main difficulty stems from the fact that the translation symmetry for the DP equation gives rise to a conserved quantity equivalent to the $L^2$-norm, which by itself can not bound the higher-order nonlinear terms in the Lagrangian. The remedy is to observe that, given a sufficiently smooth initial condition satisfying a measurable constraint,  the $L^\infty$ orbital norm of the perturbation is bounded above by a function of its $L^2$ orbital norm, yielding the orbital stability in the $L^2\cap L^\infty$ space.
\end{abstract}

\section{Introduction}
Sitting at the intersection of integrable systems and nonlinear hydrodynamic models of shallow water waves, the DP equation \cite{D-P},
\begin{equation}\label{DPk}
m_t+2ku_x+3mu_x+um_x=0,\quad x\in\R,  \;  t>0,
\end{equation}
together with 
the Korteweg-de Vries (KdV) equation \cite{KdV},
\begin{equation}
u_t+u_{xxx}+uu_x=0,
\end{equation}
and the Camassa-Holm (CH) equation \cite{CH, F-F}
\begin{equation} \label{CH}
m_t+2ku_x+2mu_x+um_x=0,
\end{equation}
where $m\triangleq u-u_{xx}$  is the momentum density and $ k > 0 $ is a parameter related to the critical shallow water speed,
has drawn much attention throughout the years. Chronologically, the KdV equation, arising in the study of various physical systems, is the prototypical example of soliton manifestation, from which the revolutionary soliton theory stemmed starting in the 1960s \cite{GGKM,Lax,ZF,ZK}. The link between these three equations was established in the same paper where the DP equation was first found: Degasperis and Procesi in 1999 \cite{D-P} showed that the KdV euqation, the CH equation and the DP equation are the only three integrable candidates passing the asymptotic integrability test in a broad family of third-order dispersive nonlinear PDEs. The link was also strengthened by Constantin and Lannes in 2009 \cite{cola} proving rigorously that both the CH equation and the DP equation are valid models of shallow-water waves over a flat bed. 

    Despite these strong similarities, they are indeed distinctively different from each other in various ways. For instance, wave breaking, a genuine nonlinear phenomena just like solitons, does not take place in the KdV equation due to the strong dispersive effect in the system, but manifests in the CH equation and the DP equation \cite{C-Ep, C-E, L-Y, Wh}. Interestingly, unlike KdV, the solitary waves of  the CH and DP equations (when $ k = 0 $) are peaked solitons (peakons) \cite{C-H-H,D-H-H, D-P}. Moreover, there are also major structural differences between the DP equation and the CH equation, such as that the DP equation admits a shock peakon, which is not found in the CH equation \cite{E-L-Y, Lu}; the isospectral problems in the Lax pair for the DP equation and the  CH equation are respectively of third-order \cite {D-H-H} and of second order \cite {CH}; and  the CH equation has geometric formulations \cite{C-K, Mi} while no such derivation is available for the DP equation. The list can go on but the difference we would like to emphasize here is that the conservation laws of the DP equations are in general weaker than those of the CH equation. As a result, any quantitative and qualitative study based upon the conservation laws, including the stability of solitons, is more subtle for the DP equation; see the work on the global existence, blow-up phenomena, and stability of peakons in the DP equation in  \cite{L-Y} for another example.

For solitons in the DP and CH equation, there are two distinctive scenarios, depending on the value of $k$.
\begin{itemize}
    \item In the limiting case of vanishing linear dispersion ($k=0$), smooth solitons degenerate into peaked solutions, called peakons, which are weak solutions and true solitons that interact via elastic collisions. Those peakons feature a  characteristic for the waves of great height
-- waves of the largest amplitude that are exact traveling wave solutions of the governing equations for irrotational water waves \cite{Cinvent,C-Eann}. The orbital stability of these peakons in the CH and DP equations has been verified respectively in \cite{C-S} and \cite{L-L,KM}. 
\item In the case of non-vanishing linear dispersion ($ k \neq 0$), while the orbital stability of smooth solitons of the CH equation is well understood by now \cite {CM,C-S3}, it is less clear for the DP equation. In fact, we established the spectral stability in our former work \cite{LLW} and the goal of this paper is to establish orbital stability of smooth solitons of the DP equation.
\end{itemize}

We first recall the existence result of smooth solitary waves established in \cite{LLW}. We call a solution $u(t,x)$ of the DP equation a \textit{solitary wave} if it is a localized \textit{traveling wave}; that is,
\[
u(t,x)=\phi(x-ct),
\]
where $c$ is a real number and $\phi:\R\to\R$ is a smooth function, satisfying the following additional properties.
\begin{itemize}
\item (Localization) $\lim\limits_{\xi\to\pm\infty}\phi(\xi)=0$.
\item (Monotonicity) There exists $\xi_0\in\R$ so that $\max\limits_{\xi\in\R}\phi(\xi)=\phi(\xi_0)$. Moreover, $\phi$ is strictly increasing on $(-\infty,\xi_0)$ and strictly decreasing on $(\xi_0,\infty)$.
\end{itemize}
An existence result of such solitary waves in \cite{LLW} is cited below for conveniences.
\begin{theorem}[existence \cite{LLW}]\label{profile}
Given the physical condition $c>2k>0$, there exists a unique $c-$speed smooth solitary wave $\phi^c(\xi)$ with its shape depending on $c$ and its maximum height $$\frac{c-2k}{4}<\Phi_{c}\triangleq \max\limits_{\xi\in\R} \{\phi^c(\xi)\}<c-2k.$$ In addition, the function $\phi^c(\xi)$ is even and strictly increasing on $(-\infty, 0)$.
\end{theorem}

When it comes to the stability of solitons in the DP equation, there are several key observations. Firstly, the DP equation has a Hamiltonian formulation. In fact, the DP equation \eqref{DPk}, after being applied with the operator $(1-\partial_x^2)^{-1}$, can be rewritten in a weak form in terms of $u$; that is,
\begin{equation}\label{kDP}
\partial_tu+\partial_x\left(\frac{1}{2}u^2+p*(\frac{3}{2}u^2+2ku)\right)=0,\quad t>0, \; \, x\in\R,
\end{equation}
where $p(x)=\frac{1}{2}e^{-\vert x\vert}$ is the impulse response corresponding to the operator $1-\partial_x^2$ so that for all $f\in L^2(\R)$,
\[
(1-\partial_x^2)^{-1}f=p*f.
\]
From now on, whenever we mention the DP equation, it is this weak form we refer to. Moreover, the DP equation can be written as an infinite dimensional Hamiltonian PDE; that is,
\begin{equation}\label{DP}
u_t=J\frac{\delta H}{\delta u}(u),
\end{equation}
where $$J\triangleq\partial_x(4-\partial_x^2)(1-\partial_x^2)^{-1},\quad H(u)\triangleq-\frac{1}{6}\int \left ( u^3+6k \left ((4-\partial_x^2)^{-\frac{1}{2}}u \right )^2 \right ) \, dx. $$
Secondly, the DP equation is translation invariant, which has two important consequences. One is that the translation invariance gives rise to a conserved quantity; that is,
\begin{equation}\label{E3DP}
S(u)=\frac{1}{2}\int_\R u\cdot(1-\partial_x^2)(4-\partial_x^2)^{-1}u\,dx;
\end{equation}
see \cite{LLW} for a more detailed discussion.
Although the DP equation admits infinitely many conserved quantities, this conserved quantity $S$ and the Hamiltonian $H$ are the essential ones relevant to our study of stability. Compared with its counterpart in the CH equation,
\begin{equation}\label{E3CH}
\widetilde{S}(u)=\int_\R \left ( u^2+u_x^2 \right ) \,dx,
\end{equation}
which is equivalent to the square of the $H^1$-norm, the conserved quantity $S$ is only equivalent to the square of the $L^2$-norm, leading to the most significant difficulties we have to overcome in order to prove stability result. The other consequence is that the spatial translation of any solitary wave $\psi_c$ generates a family of solutions, named \textit{the orbit of the solitary wave $\psi_c$} and denoted as
\[
\mathcal{M}_c=\{\phi^c(\cdot+x_0) \mid \; x_0\in\R\}.
\]
As a typical result for nonlinear dispersive PDEs with extra conserved quantities, the solitary wave $\phi^c$ is not even a critical point of the Hamiltonian. Instead, it is a critical point, but still not a local minimum, of the Lagrangian
\[
Q(u;\lambda)\triangleq H(u)+\lambda S(u)=-\frac{1}{6}\int \left ( u^3+6k \left ((4-\partial_\xi^2)^{-\frac{1}{2}}u \right )^2 \right ) \, d\xi +\frac{\lambda}{2}\int_\R u\cdot(1-\partial_\xi^2)(4-\partial_\xi^2)^{-1}u\,d\xi.
\]
As a result, the best we can hope for is \textit{orbital stability}; that is, a wave starting sufficiently close to the solitary wave $\phi^c$ remains close to the orbit of the solitary wave up to the time of existence. Indeed, the orbital stability  of solitary wave $\phi^c$ is the main result of this paper.   
\begin{theorem}[Orbital stability]\label{stability}
Assume that $c>2k>0$. The solitary wave $\phi^c(x-ct)$ of the DP equation \eqref{kDP} is orbitally stable.  More specifically,
for every $\e>0$, there is $\delta>0$ such that, for the initial value problem of the DP equation,
\begin{equation}\label{e:DP-IVP}
\begin{cases}
\partial_tu+\partial_x\left(\frac{1}{2}u^2+p*(\frac{3}{2}u^2+2ku)\right)=0,\\
u(0,x)=u_0(x),
\end{cases}
\end{equation}
with initial condition satisfying the following properties:
\begin{itemize}
    \item (Regularity) There is a positive constant $s>3/2$ such that $u_0\in H^s(\R)$. In addition, $$w_0\triangleq u_0-(u_0)_{xx}+\frac{2k}{3}>0,$$ is a positive Radon measure in the sense that the mapping $f\mapsto \int_\R fw\textrm{d}x$ gives a continuous linear functional on the space of compact-supported continuous scalar functions equipped with the canonical limit topology;
    \item (Smallness) $\Vert u_0-\phi^c\Vert_{L^2} < \delta$,
\end{itemize}
then the solution $u(t,x)$ to the initial value problem \eqref{e:DP-IVP} is a global strong one; that is, 
\[
u\in C([0,\infty), H^s(\R))\cap C^1([0,\infty),H^{s-1}(\R)),
\]
and for any $t\geq 0$, 
\[
\inf_{x_0\in\R}\Vert u(t,\cdot)-\phi^c(\cdot-x_0)\Vert_{L^2} < \e,\quad \inf_{x_0\in\R}\Vert u(t,\cdot)-\phi^c(\cdot-x_0)\Vert_{L^\infty} < O(\e^{\frac{2}{3}}).
\]
\end{theorem}
The set of initial profiles satisfying the regularity and smallness conditions in Theorem \ref{stability} is not empty. As a matter of fact, a sufficient condition for the regularity requirements of the initial data is that $u_0\in H^s(\R)$ with some $s\geq 3$ and 
\[
\Vert u_0-\phi^c\Vert_{H^3}\leq \frac{2\sqrt{2}}{3}k;
\]
see Section \ref{s:4} for details. As a result, this sufficient condition leads to the following concise version of the orbital stability result.
\begin{coro}[Orbital stability for initial data with stronger smoothness]\label{stability2}
Assume that $c>2k>0$. The solitary wave $\phi^c(x-ct)$ of the DP equation \eqref{kDP} is orbitally stable.  More specifically,
for every $\e>0$, there is $\delta\in(0, \frac{2\sqrt{2}}{3}k]$ such that, if the Cauchy problem of the DP equation \eqref{e:DP-IVP}
\[
\begin{cases}
\partial_tu+\partial_x\left(\frac{1}{2}u^2+p*(\frac{3}{2}u^2+2ku)\right)=0,\\
u(0,x)=u_0(x),
\end{cases}
\]
admits $u_0\in H^3(\R)$ with
\[
\Vert u_0-\phi^c\Vert_{H^3(\R)} < \delta\leq \frac{2\sqrt{2}}{3}k,
\]
then the solution $u(t,x)$ to the initial value problem \eqref{e:DP-IVP} is a global strong one; that is, 
\[
u\in C([0,\infty), H^s(\R))\cap C^1([0,\infty),H^{s-1}(\R)),
\]
and for any $t\geq 0$, 
\[
\inf_{x_0\in\R}\Vert u(t,\cdot)-\phi^c(\cdot-x_0)\Vert_{L^2} < \e,\quad \inf_{x_0\in\R}\Vert u(t,\cdot)-\phi^c(\cdot-x_0)\Vert_{L^\infty} < O(\e^{\frac{2}{3}}).
\]
\end{coro}
\begin{re}
The global existence of strong solutions was given in \cite{L-Y} for $k=0$. The proof for the case $k>0$ is a slight modification of the $k=0$ one and given in Section \ref{s:2}.
\end{re}
\begin{re}
There is also a global existence of weak solutions in $L^2$-space given in \cite{E-L-Y}. The regularity requirement in Theorem \eqref{stability} can be relaxed to 
\[
u_0\in L^2(\R), \quad w=u_0-(u_0)_{xx}+\frac{2k}{3} \text{ is a positive Radon measure.}
\]
The peakon case when $k=0$ can be seen in \cite{KM}.
\end{re}

The orbital stability proof follows the framework seminally developed by Grillakis, et.al. \cite{GSS, K-P}, which turns the characterization of soliton stability into the verification of the coercivity of the bilinear form of the second variational derivative of the Hamiltonian on a restrained space. The idea is to check whether the unstable directions are prohibited by constraints arising from symmetries, kernel of the skew-symmetric operator and foliation decomposition of the solution nearby the orbit of the solitary wave, so that the nonlinear wave under study becomes a constrained minimizer and thus orbitally stable. Typically, this framework requires verification of several conditions, listed below.
\begin{itemize}
 \item Bounded invertibility of the skew symmetric operator in the Hamiltonian PDE;
 \item The linear operator, corresponding to the second variational derivative of the Hamiltonian, admits certain spectral properties. Typically, the kernel should be finite-dimensional. The intersection of the spectrum and the negative axis consists of finite many negative eigenvalues and the intersection with the positive axis admits a positive distance from the origin.
 \item Convexity of the scalar function which maps the group velocity to the Lagrangian evaluated at the solitary wave and the group velocity.
 \item High order terms in the Lagrangian (third order and above) can be bounded above by a function of the conserved quantity given by the translation invariance.
 \end{itemize}

 While for the CH equation, the above steps can be tackled in a relatively standard way similar to the one for the KdV equation, it is not as straightforward for the DP equation, especially the second and last one.  The first three obstacles for the DP equation has been taken care of in our previous work \cite{LLW}. As for the last one, the conserved quantity $\widetilde{S}$ naturally lifts the $L^2$ coercivity of the $CH$ equation to an $H^1$ one so that high order terms in the Lagrangian of the CH equation can readily be bounded by high order terms in the $H^1$-norm \cite{CM}. Unfortunately, the high order term, say $\int h^3\,dx$, can not generically be bounded above solely by the $L^2$-norm of $h$. Inspired by the recent work of Khorbatly and Molinet \cite{KM}, we find a remedy to control the $L^\infty$-norm with the $L^2$-norm via imposing additional regularity on the initial condition and thus take care of the last obstacle to give a complete proof of the orbital stability theorem. We also remark that for the case of null-linear dispersion, the uniform $L^\infty$ control is not need. Instead, the control of a point distance is enough, say the difference between the peak of the peakon and that of the perturbation \cite{L-Y}. 

The remainder of the paper is organized as follows. In section \ref{s:2} we discuss the well-posedness of the DP equation and prove an \textit{a priori} estimate.  In Section \ref{s:4} we give the proof of Theorem \ref{stability}.

\bigskip

\section{Well-posedness and \emph{a priori} estimates}\label{s:2}
The well-posedness of the initial value problem serves as the precondition of any qualitative study of the dynamics. For clarity, we list several well-posedness results of the DP equation with $k>0$, together with some global existence results and finite-time blow up phenomena, whose proofs are based upon the ones for the vanishing linear dispersion $(k=0)$ case without or with mild modifications.

A local well-posedness result for the Cauchy problem \eqref{e:DP-IVP} with $k=0$ is obtained in \cite{Y1} via applying Kato's theorem \cite{K}. With exactly the same argument, we have the following local well-posedness result for the Cauchy problem \eqref{e:DP-IVP} with $k>0$.

\begin{prop}[Uniqueness and local existence of strong solutions]\label{p:loc}
Given the initial profile $u_0\in H^s(\R)$ with $s>\frac{3}{2},$ there exist a maximal time $T=T(u_0)\in(0,\infty]$, independent of the choice of $s$, and a unique solution $u$ to the Cauchy problem \eqref{e:DP-IVP} such that 
\[
u=u(\cdot;u_0)\in C([0,T); H^s(\R))\cap C^1([0,T); H^{s-1}(\R)).
\]
Moreover, the solution depends continuously on the initial data and is called \emph{a strong solution} due to its smoothness. 
\end{prop}
\begin{re}
Except for that $f(u)=-\partial_x(1-\partial_x^2)^{-1}(\frac{3}{2}u^2)$ in Lemma 2.4 in \cite{Y1} becomes $f(u)=-\partial_x(1-\partial_x^2)^-1(\frac{3}{2}u^2+2ku)$, the proof of Proposition \ref{p:loc} is the same as the one of Theorem 2.2 in \cite{Y1} and thus omitted.
\end{re}

Furthermore, the strong solution is a global one if the initial condition is sufficiently ``regular". More specifically, we have the following global existence result.
\begin{prop}[Global existence of strong solutions] \label{linfty}
Given that the initial profile $u_0\in H^s(\R)$ with $s>\frac{3}{2}$ and $w=u_0-u_{0,xx}+\frac{2}{3}k$ is a Radon measure of fixed sign, the strong solution to the Cauchy problem \eqref{e:DP-IVP} then exists globally in time; that is,
\[
u=u(\cdot;u_0)\in C([0,\infty); H^s(\R))\cap C^1([0,\infty); H^{s-1}(\R)),
\]
which admits the following additional estimates.
\begin{enumerate}
    \item The magnitude of $u_x$ is bounded above by the sum of the magnitude of $u$ and the constant $\frac{2k}{3}$. As a matter of fact, we have, for all $(t,x)\in[0,\infty)\times\R$,
    \begin{equation}\label{e:uxu}
        \vert u_x(t,x)\vert\leq\vert u(t,x)+\frac{2}{3}k\vert.
    \end{equation}
    \item The $L^\infty$ norm of $u$ is bounded. More specifically, we have, for all $t\in[0,\infty)$,
    \begin{equation}\label{e:uinfty}
        \Vert u(t,\cdot)\Vert_{L^\infty}\leq  \sqrt{2}(1+\sqrt{2})\Vert u_0\Vert_{L^2(\R)}+\frac{4}{3}k.
    \end{equation}
\end{enumerate}
\end{prop}

To prove the proposition, we first introduce several convenient results, whose proofs for the case $k>0$ are straightforward adaptions from the ones for the case $k=0$ established in \cite{Y1, Y2} and thus omitted. The first result \cite[Theorem 3.1]{Y1} shows that the only way undermining the global existence of a strong solution is that the wave blows up in finite time.
\begin{lemma}[Wave breaking]\label{blowup}
Given the initial profile $u_0\in H^s(\R)$ with $s>\frac{3}{2},$ the blowup of the strong solution $u=u(\cdot,u_0)$ in finite time $T<+\infty$ occurs if and only if 
\[
\liminf_{t\to T^-}\{\inf_{x\in\R}[u_x(t,x)]\}=-\infty.
\]
\end{lemma}

The second result \cite[Lemma 3.1\&3.2]{Y2} shows that given sufficiently smooth initial data, the profile
\begin{equation}\label{e:w}
    w(t,x)\triangleq m(t,x)+\frac{2k}{3},
\end{equation}
is of fixed sign if $w_0=m_0(x)+\frac{2k}{3}$ does.
\begin{lemma}[Sign preserving of $w$]\label{unisign}
Let the initial profile $u_0\in H^s(\R)$ with $s\geq 3$, and let $T>0$ be the maximal existence time of the strong solution $u$ to the Cauchy problem \eqref{e:DP-IVP}. We then have
\[
w(t,x)q_x^3\equiv w_0(x),\quad\forall (t,x)\in[0,T)\times\R,
\]
where $q\in C^1([0,T)\times\R,\R)$ is the unique solution to the differential equation
\begin{equation}\label{path}
 \left\{\begin{array}{ll}
q_t=u(t,q),\quad &t\in[0,T),\\
q(0,x)=x,\quad & x\in \R ,\\
 \end{array}
 \right.
 \end{equation}
 with $q_x>0$ for all $(t,x)\in[0,T)\times\R$.
\end{lemma}
Exploiting Lemma \ref{blowup} and Lemma \ref{unisign}, we now give the proof of Proposition \ref{linfty}.
\begin{proof}[Proof of Proposition \ref{linfty}] 
We only prove the proposition for the case $s=3$ and $w_0\geq0$, thanks to the density argument. Let $T$ be the maximal time of existence of the strong solution $u$ to the Cauchy problem \eqref{e:DP-IVP}. As a result of Lemma \ref{unisign}, we conclude from $w_0\geq0$ that
\[
w(t,x)\geq0,\quad \forall (t,x)\in[0,T)\times\R,
\]
which, thanks to the convolution $u+\frac{2k}{3}=p*(m+\frac{2k}{3})=p*w$, yields 
\[
u(t,x)\geq0,\quad \forall (t,x)\in[0,T)\times\R.
\]
Moreover, the expansion of the convolution $u+\frac{2k}{3}=p*(m+\frac{2k}{3})$ takes the form
\[
u(t,x)+\frac{2}{3}k=\frac{e^{-x}}{2}\int_{-\infty}^xe^\eta [m(t,\eta)+\frac{2}{3}k]d\eta+\frac{e^x}{2}\int_x^\infty e^{-\eta}[m(t,\eta)+\frac{2}{3}k]d\eta,\]
whose partial derivative with respect to $x$ reads
\[
u_x(t,x)=-\frac{e^{-x}}{2}\int_{-\infty}^xe^\eta [m(t,\eta)+\frac{2}{3}k]d\eta+\frac{e^x}{2}\int_x^\infty e^{-\eta}[m(t,\eta)+\frac{2}{3}k]d\eta.
\]
The summation and difference of these two equations show that
$$
\begin{cases}
&u(t,x)+\frac{2}{3}k+u_x(t,x)=e^x\int_x^\infty e^{-\eta}[m(t,\eta)+\frac{2}{3}k]d\eta\geq0,\\
&u(t,x)+\frac{2}{3}k-u_x(t,x)=e^{-x}\int_{-\infty}^x e^{\eta}[m(t,\eta)+\frac{2}{3}k]d\eta\geq0,
\end{cases}
$$
from which we deduce the inequality \eqref{e:uxu}; that is, for all $(t,x)\in[0,T)\times\R$,
\[
\vert u_x(t,x)\vert\leq\vert u(t,x)+\frac{2}{3}k\vert.
\]
We are now left to show that the $L^\infty$ norm of $u(t,\cdot)$ is bounded and the solution globally exists. To do that, we fix $(t,x)\in [0,T)\times\R$ and denote by $[x]$ the integer part of $x$. Due to that $u(t,\cdot)\in H^3(\R)\subset C(\R)$, the mean value theorem ensures that there exists $\bar{x}\in[[x]-1,[x]]$ such that 
\[
u^2(t,\bar{x})=\int_{[x]-1}^{[x]}u^2(t,\eta)d\eta\leq
\Vert u(t,\cdot)\Vert_{L^2(\R)}^2\leq 8S(u(t,\cdot))=8 S(u_0)\leq 4\Vert u(t,\cdot)\Vert_{L^2(\R)}^2,
\]
where we apply the fact that the conserved quantity $S(u)=\frac{1}{2}\int_\R u\cdot(1-\partial_x^2)(4-\partial_x^2)^{-1}u\,dx$ satisfies
\begin{equation}\label{e:SL2}
    \frac{1}{8}\Vert u(t,\cdot)\Vert_{L^2(\R)}^2\leq S(u)\leq \frac{1}{2}\Vert u(t,\cdot)\Vert_{L^2(\R)}^2.
\end{equation}
Therefore, we have $|u(t,\bar{x})|\leq 2\Vert u(t,\cdot)\Vert_{L^2(\R)},$
which, together with \eqref{e:uxu} and $0\leq x-\bar{x}\leq 2$, yields
\begin{equation*}
\begin{aligned}
u(t,x)=u(t,\bar{x})+\int_{\bar{x}}^x u_{\eta}(t,\eta)d\eta
&\geq-2\Vert u_0\Vert_{L^2(\R)}-\int_{\bar{x}}^x(\vert u(t,\eta)\vert+\frac{2}{3}k)d\eta\\
&\geq-2\Vert u_0\Vert_{L^2(\R)}-\frac{4}{3}k-\sqrt{2}\Vert u(t,\cdot)\Vert_{L^2(\R)}\\
&=-\sqrt{2}(1+\sqrt{2})\Vert u_0\Vert_{L^2(\R)}-\frac{4}{3}k.
\end{aligned}
\end{equation*}
To confirm the estimate \eqref{e:uinfty}, we prove by contradiction and suppose that there exists $x_*$ such that 
\[u(t,x_*)>\sqrt{2}(1+\sqrt{2})\Vert u_0\Vert_{L^2(\R)}+\frac{4}{3}k.\]
On one hand, the mean-value theorem ensures that there exists $\bar{x}_*\in [[x_*]+1,[x_*]+2]$ such that
\[
u^2(t,\bar{x}_*)=\int_{[x_*]+1}^{[x_*]+2}u^2(t,\eta)d\eta\leq
\Vert u(t,\cdot)\Vert_{L^2(\R)}^2\leq 8S(u(t,\cdot))=8 S(u_0)\leq 4\Vert u(t,\cdot)\Vert_{L^2(\R)}^2.
\]
On the other hand, we have
\begin{equation*}
\begin{aligned}
u(t,\bar{x}_*)=u(t,x_*)+\int_{x_*}^{\bar{x}_*}u_{\eta}(t,\eta)d\eta
&>\sqrt{2}(1+\sqrt{2})\Vert u_0\Vert_{L^2(\R)}+\frac{4}{3}k-\int_{x_*}^{\bar{x}_*}(\vert u(t,\eta)\vert+\frac{2}{3}k)d\eta\\
&\geq \sqrt{2}(1+\sqrt{2})\Vert u_0\Vert_{L^2(\R)}+\frac{4}{3}k-(\frac{4}{3}k+\sqrt{2}\Vert u(t,\cdot)\Vert_{L^2(\R)})\\
&=2\Vert u_0\Vert_{L^2(\R)}.
\end{aligned}
\end{equation*}
The fact that the above two estimates contradict each other completes the proof of the estimate \eqref{e:uinfty}. Estimate \eqref{e:uxu} and \eqref{e:uinfty} then imply that $u_x$ is uniformly bounded for all $(t,x)\in[0,T)\times\R$, which together with Lemma \ref{blowup}, ensures the global existence of the strong solution and thus concludes the proof. 
\end{proof}

To bound the high order term $\int_\R u^3dx$ in the Lagrangian with the conserved quantity $S(u)$ whose square root is equivalent to the $L^2$-norm of $u$, we take advantage of the argument in \cite[Lemma 2]{KM} to derive the following \emph{a priori} estimate.
\begin{prop}[\emph{a priori} $L^\infty$-$L^2$ estimate]\label{hlinfty}
Let $\psi\in W^{1,\infty}\cap L^2(
R)$ and the initial data $u_0\in H^s(\R)$ with $s>\frac{3}{2}$ and $w_0=m_0+\frac{2k}{3}$ a Radon measure of fixed sign. The difference between the strong solution $u$ to the Cauchy problem \eqref{e:DP-IVP} and the function $\psi$, denoted as $g(t,x)\triangleq u(t,x)-\psi(x)$, admits the following estimate
\begin{equation}\label{e:Linfty}
\Vert g(t,\cdot)\Vert_{L^\infty(\R)}\leq
\Vert g(t,\cdot)\Vert_{L^2(\R)}^{2/3}\bigg(1+\frac{4}{3}k+\sqrt{2}\Vert g(t,\cdot)\Vert_{L^2(\R)}^{2/3}+2\Vert \psi\Vert_{L^\infty(\R)}+2\Vert \psi'\Vert_{L^\infty(\R)}\bigg), \quad\forall t\in[0,\infty).
\end{equation}
\end{prop}
\begin{proof}
Fix $t\in[0,\infty)$, we denote $\al=\Vert g(t,\cdot)\Vert_{L^2(\R)}^{2/3}$ and assume $\alpha>0$, due to the fact that the case $\alpha=0$ makes both sides of \eqref{e:Linfty} zero. Fixing $x\in\R$, there exists $k\in\mathbb{Z}$ such that 
$x\in[k\al,\,(k+1)\al)$. By the mean value theorem, there exists 
$\bar{x}\in[(k-1)\al,\,k\al]$ such that
\[
g^2(t,\bar{x})=\frac{1}{\al}\int_{(k-1)\al}^{k\al}g^2(t,\eta)d\eta\leq\frac{1}{\al}\Vert g(t,\cdot)\Vert_{L^2(\R)}^2=\al^2,
\]
which, together with Proposition \ref{linfty} and that $0\leq x-\bar{x}\leq 2\al$, yields
\begin{equation}\label{ppp}
\begin{aligned}
g(t,x)&=g(t,\bar{x})+\int_{\bar{x}}^x g_\eta(t,\eta)d\eta\\
&\geq-\al-\frac{4\al}{3}k-\sqrt{2\al}\left\Vert\bigg(\vert u(t,\cdot)\vert +\vert \psi'\vert \bigg)\right\Vert_{L^2([(k-1)\al,(k+1)\al])}\\
&\geq-\al-\frac{4\al}{3}k-\sqrt{2\al}\left\Vert\bigg(\vert g(t,\cdot)\vert +\vert \psi\vert +\vert \psi'\vert \bigg)\right\Vert_{L^2([(k-1)\al,(k+1)\al])}\\
&\geq-\al-\frac{4\al}{3}k-\sqrt{2\al}\left[\Vert g(t,\cdot)\Vert_{L^2(\R)}+\sqrt{2\al}\bigg(\Vert \psi\Vert_{L^\infty(\R)}+\Vert \psi'\Vert_{L^\infty(\R)}\bigg)\right]\\
&=-\al\bigg(1+\frac{4}{3}k+\sqrt{2}\al+2\Vert \psi\Vert_{L^\infty(\R)}+2\Vert \psi'\Vert_{L^\infty(\R)}\bigg).
\end{aligned}
\end{equation}
Similarly to the argument in Proposition \ref{linfty}, we use the proof by contradiction and suppose that there exists $x_*\in\R$ such that 
\[
g(t,x_*)>\al\bigg(1+\frac{4}{3}k+\sqrt{2}\al+2\Vert \psi\Vert_{L^\infty(\R)}+2\Vert \psi'\Vert_{L^\infty(\R)}\bigg).
\]
Then there exists $k_*\in \R$ with $x_*\in [k_*\al,(k_*+1)\al)$ 
such that by the mean value theorem, there exists $\bar{x}_*\in[(k_*+1)\al, (k_*+2)\al]$  such that, on one hand,
\[
g^2(t,\bar{x}_*)=\frac{1}{\al}\int_{(k_*+1)\al}^{(k_*+2)\al}g^2(t,\eta)d\eta\leq \al^2.
\]
On the other hand, proceeding as in \eqref{ppp}, we have
\begin{equation*}
\begin{aligned}
u(t,\bar{x}_*)-\psi(\bar{x}_*)
=&u(t,x_*)-\psi(x_*)+\int_{x_*}^{\bar{x}_*}[u_\eta(t,\eta)-\psi'(\eta)]d\eta\\
>&\al\bigg(1+\frac{4}{3}k+\sqrt{2}\al+2\Vert \psi\Vert_{L^\infty(\R)}+2\Vert \psi'\Vert_{L^\infty(\R)}\bigg)-\\
&\frac{4\al}{3}k-\sqrt{2\al}\left\Vert\bigg(\vert u(t,\cdot)\vert +\vert \psi'\vert \bigg)\right\Vert_{L^2([k\al,(k+2)\al])}\\
\geq &\al\bigg(1+\frac{4}{3}k+\sqrt{2}\al+2\Vert \psi\Vert_{L^\infty(\R)}+2\Vert \psi'\Vert_{L^\infty(\R)}\bigg)-\\
&\frac{4\al}{3}k-\al\left(\sqrt{2}\al+2\Vert \psi\Vert_{L^\infty(\R)}+2\Vert \psi'\Vert_{L^\infty(\R)}\right)\\
=&\al.
\end{aligned}
\end{equation*}
Again, the incompatibility of the above two estimates concludes the proof of the proposition.
\end{proof}

\bigskip

\section{Orbital Stability of Degasperis-Procesi Solitons}\label{s:4}
In this  section, we give proofs of Theorem \ref{stability} and \ref{stability2} in order.  The former one is based on  the frame work of Grillakis, \textit{et.al.} \cite{GSS, K-P} with major modifications on nonlinear estimates, while the latter one is a consequence of Theorem \eqref{stability}. To start with, we recall that the DP equation \eqref{kDP} is Hamiltonian in the form of \eqref{DP}; that is,
\[
u_t=J\frac{\delta H}{\delta u}(u),
\]
where $$J\triangleq\partial_x(4-\partial_x^2)(1-\partial_x^2)^{-1},\quad H(u)\triangleq-\frac{1}{6}\int \left ( u^3+6k \left ((4-\partial_x^2)^{-\frac{1}{2}}u \right )^2 \right ) \, dx, $$
whose translation symmetry gives rise to the conserved quantity
\[
S(u)=\frac{1}{2}\int(1-\partial_x^2)(4-\partial_x^2)^{-1}u\cdot udx.
\]
The smooth solitary wave $\phi^c$ of the DP equation is a critical point of the Lagrangian
\[
Q(u;c)=H(u)+cS(u).
\]
We now introduce the scalar function 
\begin{equation}
    \begin{matrix}
    R: & (2k,\infty) & \longrightarrow & \R \\
    & c & \longmapsto & Q(\phi^c,c)
    \end{matrix}
\end{equation}
which is shown to be strictly convex in \cite[Lemma 4.2]{LLW}. 
\begin{lemma}[Convexity, \cite{LLW}]\label{l:convex}
The function $R$ is strictly convex in the sense that 
\begin{equation}\label{e:convex}
 R^{\prime\prime}(c)=\frac{d}{d c} \bigg(S(\phi^c)\bigg)>0, \quad \forall c>2k.   
\end{equation}
\end{lemma}
Thanks to the equivalence between $L^2$-norm and the square root of $S$ and the fact that $S(\phi^c)$ is strictly monotonic, we only need to prove Theorem \ref{stability} under the conservation constraint
\[
S(u)=S(\phi^c).
\]
\begin{re}
We briefly explain why it is sufficient to give the proof for the special case $S(u)=S(\phi^c)$. Given that Theorem \ref{stability} holds when $S(u)=S(\phi^c)$ for any $c>2k$, we fix some $c_*>2k$ and have that, for any $\epsilon>0$, 
\begin{itemize}
    \item there exists $c_0>0$ so that
    \[
    \|\phi^c-\phi^{c^*}\|_{L^2(\R)}\leq \epsilon/2;
    \]
    \item for any $c\in[c^*-c_0,c^*+c_0]$, there exists $\delta(\epsilon/2, c)$ so that Theorem \ref{stability} holds for $S(u)=S(\phi^c)$ with $\Vert u_0-\phi^c \Vert_{L^2}<\delta(\epsilon/2, c)\implies \inf\limits_{x_0\in\R}\Vert u(t,\cdot)-\phi^c(\cdot-x_0)\Vert_{L^2(\R)}<\frac{\epsilon}{2}$.
\end{itemize}
As a result, taking 
\[
\delta_*\triangleq\min\{\frac{\epsilon}{2}, \min_{|c-c_*|<c_0}\{\delta(\frac{\epsilon}{2},c)\}\}>0,
\]
we have that $\Vert u_0-\phi^{c_*} \Vert_{L^2}<\delta_*$ leads to $\inf\limits_{x_0\in\R}\Vert u(t,\cdot)-\phi^{c_*}(\cdot-x_0)\Vert_{L^2(\R)}<\epsilon$, which concludes the proof for the general case.
\end{re}
For convenience, we from now on fix $c>2k$, suppress the supper index of $\phi^c$ and also introduce a local foliation of a neighborhood of the orbit $\mathcal{M}_c=\{\phi(\cdot+x_0\mid x_0\in\R)\}$. More specifically, there exists $\delta_1>0$ such that for any 
\[
u\in\mathcal{N}_c\triangleq \{v\in L^2(\R)\mid \inf\limits_{x_0\in\R}\Vert v(\cdot)-\phi(\cdot-x_0)\Vert_{L^2(\R)}<\delta_1\},
\]
there exists a unique foliation decomposition 
\[
u=\cat(r)\bigg(\phi+h \bigg),
\]
where $\cat(r)u(\cdot)\triangleq u(\cdot+r)$ is the translation operator and $h\in L^2(\R)$ is perpendicular to $\partial_x\phi$; that is $(h,\partial_x\phi)\triangleq\int_\R h\partial_x\phi dx=0$.If the initial data falls in the neighborhood $\mathcal{N}_c$; that is,
\[
\|u_0-\phi\|_{L^2(\R)}<\delta_1,
\]
there exist a maximal time $T_m>0$ such that the strong solution $u$ stays within $\mathcal{N}_c$ for $t\in[0,T_m)$; that is,
\[
T_m\triangleq \max_{T\geq0}\{T\mid u(t,\cdot)\in\mathcal{N}_c, \forall t\in[0,T)\}>0.
\]
As a result, for $t\in[0,T_m)$, the strong solution admits the foliation decomposition 
\[
u(t,x)=\cat(r(t))\bigg(\phi(x)+h(t, x)\bigg),
\]
where $(h(t,\cdot),\partial_x\phi)=0$. 

We now introduce the time-invariant quantity
\[
\overline{Q_c}\triangleq Q_c(u)-Q_c(\phi),
\]
whose expansion in terms of $h$ admits the expression
\begin{equation}\label{QExp}
\overline{Q_c}=Q_c(\phi(x)+h(t, x))-Q_c(\phi)=\frac{1}{2}(L_ch,h)-\frac{1}{6}\int h^3 d \xi,
\end{equation}
where
\begin{equation}\label{e:Lc}
    L_c\triangleq\frac{\delta^2 Q_c}{\delta u^2}(\phi)
=c-\phi-(3c+2k)(4-\partial_x^2)^{-1},
\end{equation}
and $h(t,\cdot)$ lies in the nonlinear admissible set
\[
\caa\triangleq\{h\in L^2(\R)\mid S(h+\phi)=S(\phi), (h,\partial_x \phi)=0\}.
\]
For the sake of narrative coherency, we give the following proposition and relegate its proof to the end of the section.

\begin{prop}\label{ineq}
For sufficiently small $h\in\caa$, there exist $\alpha, \beta>0$ such that 
\begin{equation}\label{Lbnd}
\frac{1}{2}(L_ch,h)\geq \alpha \|h\|^2_{L^2(\R)}-\beta\|h\|^3_{L^2(\R)}.
\end{equation}
\end{prop}

We now give the proof of Theorem \ref{stability}.
\begin{proof}[Proof of Theorem \ref{stability}] 
For convenience, we assume $\delta<1$. We first derive an upper bound of $|\overline{Q_c}|$ in terms of the $L^2$ norm of $ h_0(x)\triangleq h(0,x)$. Combining the expansion \eqref{QExp} and the expression of $L_c$ as in \eqref{e:Lc}, we have 
\begin{equation}\label{e:upbd1}
    \begin{aligned}
|\overline{Q_c}|
&\leq c\|h_0\|_{L^2(\R)}^2
+\frac{1}{6}\Vert h_0\Vert_{L^\infty(\R)}\cdot\Vert h_0\Vert_{L^2(\R)}^2.
\end{aligned}
\end{equation}
According to Proposition \ref{hlinfty}, the $L^\infty$ norm of $h$ admits the following estimate
\begin{equation}\label{e:Qcinq2}
    \begin{aligned}
\Vert h(t,\cdot)\Vert_{L^\infty(\R)}=&\Vert \cat(r(t))h(t,\cdot)\Vert_{L^\infty(\R)}=\Vert u(t,\cdot)-\cat(r(t))\phi(\cdot)\Vert_{L^\infty}\\
\leq& \Vert h(t,\cdot)\Vert_{L^2(\R)}^{2/3}\bigg(1+\frac{4}{3}k+\sqrt{2}\Vert h(t,\cdot)\Vert_{L^2(\R)}^{2/3}+2\Vert \phi\Vert_{L^\infty(\R)}+2\Vert \phi'\Vert_{L^\infty(\R)}\bigg).
\end{aligned}
\end{equation}
The estimate \eqref{e:Qcinq2} for $t=0$, plugged into the estimate \eqref{e:upbd1}, leads to 
\begin{equation}\label{e:upbd}
    \begin{aligned}
    |\overline{Q_c}|\leq c\|h_0\|_{L^2(\R)}^2+\gamma\|h_0\|_{L^2(\R)}^{8/3}+\frac{\sqrt{2}}{6}\|h_0\|_{L^2(\R)}^{10/3}<K\delta^2,
    \end{aligned}
\end{equation}
where $\gamma(c,k)\triangleq \frac{1}{6} \big(1+\frac{4}{3}k+2\Vert \phi\Vert_{L^\infty(\R)}+2\Vert \phi'\Vert_{L^\infty(\R)}\big)$ and $K\triangleq\max\{c,\gamma, \frac{\sqrt{2}}{6}\}$.
Similarly, we also derive a lower bound of $|\overline{Q_c}|$ in terms of the $L^2$ norm of $ h_0(x)\triangleq h(0,x)$. We first conclude form the expansion \eqref{QExp} and the inequality \eqref{Lbnd} that
\begin{equation}\label{e:Qcinq1}
    \begin{aligned}
|\overline{Q_c}|
&\geq \alpha\|h\|_{L^2(\R)}^2-\beta\|h\|^3_{L^2(\R)}
-\frac{1}{6}\Vert h\Vert_{L^\infty(\R)}\cdot\Vert h\Vert_{L^2(\R)}^2.
\end{aligned}
\end{equation}
which, together with the inequality \eqref{e:Qcinq2}, yields that
\begin{equation}\label{Qcinq}
\begin{aligned}
|\overline{Q_c}|
&\geq \alpha\|h\|_{L^2(\R)}^2-\beta\|h\|^3_{L^2(\R)}
-\frac{1}{6}\Vert h\Vert_{L^2(\R)}^{8/3}\bigg(1+\frac{4}{3}k+\sqrt{2}\Vert h\Vert_{L^2(\R)}^{2/3}+2\Vert \phi\Vert_{L^\infty(\R)}+2\Vert \phi'\Vert_{L^\infty(\R)}\bigg)\\
&=\alpha\|h\|_{L^2(\R)}^2-\gamma\|h\|_{L^2(\R)}^{\frac{8}{3}}-\beta\|h\|^3_{L^2(\R)}
-\frac{\sqrt{2}}{6}\|h\|_{L^2(\R)}^{\frac{10}{3}},
\end{aligned}
\end{equation}
For small $|\overline{Q_c}|$, the function 
\[
f(r)\triangleq |\overline{Q_c}|-\alpha r^2+\gamma r^{\frac{8}{3}}+\beta r^3
+\frac{\sqrt{2}}{6}r^{\frac{10}{3}}
\]
admits two consecutive positive roots
\[
0<r_1=\mathcal{O}(|\overline{Q_c}|^{1/2})<r_2=\mathcal{O}(1),
\]
which, together with the estimate \eqref{e:upbd}, shows that 
\[
r_1=\mathcal{O}(\delta).
\]
As a result, there exists $\delta_0\in(0,1)$ such that \[
r_1<\min\{\epsilon,\delta_1,\frac{1}{2}\left(\frac{\alpha}{\gamma}\right)^{3/2}\}<r_2.\] Furthermore,
we conclude from the inequality \eqref{Qcinq} and the continuity of $h(t)$ that if $\|h_0\|_{L^2(\R^2)}\in(0, r_1)$, then $\|h(t,\cdot)\|_{L^2(\R^2)}\in(0, r_1)$ holds globally for $t\in[0,\infty)$. 
Therefore,  for any $\varepsilon >0$, we can choose $\delta=\delta_0$ such that if
\[
\|u_0-\phi\|_{L^2(\R)}=\|h_0\|_{L^2(\R)}\leq \delta,
\]
then
 \[
 \inf\limits_{r\in\R}\|u(t,\cdot)-\cat(r)\phi\|_{L^2(\R)}=\|h(t)\|_{L^2(\R)}<r_1<\varepsilon,\quad \forall t\in[0,\infty).
 \]
 Noting that the $L^\infty$ estimate in the theorem follows from the $L^2$ estimate and Proposition \ref{hlinfty}, we conclude the proof of Theorem \ref{stability}.
\end{proof}

Exploiting the properties of the solitary wave $\phi$, Theorem \ref{stability2} naturally leads to Corollary \ref{stability2} whose proof is given below.
\begin{proof}[Proof of Corollary \ref{stability2}]Based on Theorem \ref{stability}, the proof of Theorem \ref{stability2} boils down to verify that if $u_0\in H^3(\R)$ with $\Vert u_0\phi\Vert_{H^3(\R)}<\frac{2\sqrt{2}}{3}k$, then $w_0=u_0-u_{0xx}+\frac{2k}{3}$ is a positive Radon measure. The fact that $w_0$ is a Radon measure follows directly from $u_0\in H^3(\R)$ and thus we are left to show that $w_0$ is positive everywhere. In order to do that, we recall that the solitary wave $\phi$ is a critical point of the Lagrangian $Q_c=H+cS$; that is,
\[
\frac{\delta Q_c}{\delta u}(\phi)=0,
\]
which admits the expression
\begin{equation}\label{traveling}
-[\frac{1}{2}\phi^2+(4-\partial_x^2)^{-1}2k\phi]+c(1-\partial_x^2)(4-\partial_x^2)^{-1}\phi=0.
\end{equation}
Applying $(4-\partial_x^2)$ to both sides, equation \eqref{traveling} becomes a second-order ODE
\begin{equation}\label{e:tODE}
    (c-\phi)(\phi-\phi_{xx})=\phi^2+2k\phi-\phi_x^2,
\end{equation}
which admits a first integral
\[
    \Phi(\phi,\phi_x)=k\phi^2(\frac{2}{3}\phi-c)+\frac{1}{2}(c-\phi)^2(\phi^2-\phi_x^2).
\]
A straightforward analysis \cite{LLW} shows that the solitary wave $\phi$ is in fact a component of the level set $\Phi(\phi,\phi_x)=0$, yielding 
\begin{equation}\label{e:1int}
    \phi_x^2=\frac{2k\phi^2(\frac{2}{3}\phi-c)}{(c-\phi)^2}+\phi^2.
\end{equation}
Combining \eqref{e:tODE} and \eqref{e:1int}, we have
\[
\phi-\phi_{xx}=\frac{2k}{3}\left(\frac{c^3}{(c-\phi)^3}-1\right)>0.
\]
It follows that if $\Vert u_0-\phi\Vert_{H^3}\leq \frac{2\sqrt{2}}{3}k$, then 
\[
\vert(u_0-u_{0,xx})-(\phi-\phi_{xx})\vert\leq\frac{1}{\sqrt{2}}\Vert u_0-\phi\Vert_{H^3}\leq\frac{2k}{3},
\]
and $w_0=u_0-u_{0xx}+\frac{2k}{3}\geq \phi-\phi_{xx}>0$, which concludes the proof. 
\end{proof}
We are now left to prove Proposition \ref{ineq}. In order to do so, we denote $$\widetilde{\psi}\triangleq(1-\partial_x^2)(4-\partial_x^2)^{-1}\phi,$$ and introduce the linear admissible space
\[
\caa^\prime\triangleq \{h\in L^2(\R)\mid (h,\widetilde{\psi})=(h,\partial_x\phi)=0\}.
\]
It is then straightforward to see that any $\widetilde{h}\in\caa$ with $\|\widetilde{h}\|_{L^2(\R)}$ sufficiently small admits the decomposition
\[
\widetilde{h}=h+a\phi,
\]
where $h\in \caa^\prime$ with $\Vert h\Vert_{L^2(\R)}=\mathcal{O}(\|\widetilde{h}\|_{L^2(\R)})$ and $|a|=\mathcal{O}(\|\widetilde{h}\|^2_{L^2(\R)})$. As a result, 
\[
(L_c\widetilde{h},\widetilde{h})=(L_c(h+a\phi),h+a\phi)=(L_ch,h)+\mathcal{O}(\|\widetilde{h}\|^3_{L^2(\R)}).
\]
Therefore, to prove Proposition \ref{ineq}, it suffices to prove the following lemma.
\begin{lemma}\label{bilin}
For $h\in\caa^\prime$, there exists $\alpha>0$ such that 
\[
(L_ch,h)\geq  \alpha\|h\|_{L^2}^2.
\]
\end{lemma} 
The proof of this lemma relies essentially on the spectral properties of the operator $L_c$, which has been proved in our previous work \cite{LLW} and cited below.
\begin{theorem}\cite{LLW}\label{spectrum}
The spectrum set of the operator $L_c: L^2(\R)\to L^2(\R)$, denoted as $\sigma(L_c)$, admits the following properties.
\begin{enumerate}
\item The spectrum set $\sigma(L_c)$ lies on the real line; that is, $\sigma(L_c)\subset \R$.
\item $0$ is a simple eigenvalue of $L_c$ with $\phi_x$ as its eigenfunction.
\item On the negative axis $(-\infty, 0)$, the spectrum set $\sigma(L_c)$ admits nothing but only one simple eigenvalue, denoted as $\lambda_*$, with its corresponding normalized eigenfunction, denoted as $\phi_*$.
\item The set of essential spectrum $\sigma_{ess}(L_c)$ lies on the positive real axis, admitting a positive distance to the origin.
\end{enumerate} 
\end{theorem}
We now give the proof of Lemma \ref{bilin}.
\begin{proof}[Proof of Lemma \ref{bilin}] We first introduce the projection
\[
\Pi u= u-\frac{(u, \widetilde{\psi} )}{(\widetilde{\psi},\widetilde{\psi})}\widetilde{\psi}.
\]
According to property (2) in Theorem \ref{spectrum}, the constrained operator 
\[
L_c^\Pi\triangleq \Pi L_c: \caa^\prime\to\caa^\prime 
\]
is self-adjoint and bounded invertible, and thus admits only real spectra. Noting that 
\[
(L_ch,h)=(L_c^\Pi h, h), \quad \forall h\in\caa^\prime,
\]
the proof of the lemma boils down to show that there exists $\delta_0 >0$ such that 
\[
\sigma(L_c^\Pi)\subseteq [\delta_0,\infty).
\]
Note that the essential spectrum of $L_c^\Pi$ and $L_c$ are the same, due to the fact that $L_c^\Pi-L_c$ restricted to $\caa^\prime$ is rank one and thus compact. According to property (4) of Theorem \ref{spectrum}, there exists $\delta_1>0$ such that 
\[
\sigma_{{\rm ess}}(L_c^\Pi)=\sigma_{{\rm ess}}(L_c)\subseteq (\delta_1,\infty).
\]
We now just need to show that the smallest eigenvalue of $L_c^\Pi$, called the ground-state eigenvalue and denoted as $\widetilde{\lambda}_*$, if there is any, is strictly positive. According to property (3) of Theorem \ref{spectrum}, $\lambda_*$ is the ground-state eigenvalue of the full linear operator $L_c$ with its eigenfunction $\phi_*$, yielding
\[
\widetilde{\lambda}_*=\inf\limits_{u\in\caa^\prime}\frac{(L_cu, u)}{(u,u)}\geq \inf\limits_{u\in L^2}\frac{(L_cu, u)}{(u,u)}=\lambda_*.
\]
As a matter of fact, we have an improved estimate
\[
\widetilde{\lambda}_*>\lambda_*,
\]
thanks to the fact that $\phi_*\not\in\caa^\prime$, which, in turn, is a natural consequence of the fact that
\[
(\phi_*,\widetilde{\psi})\neq0.
\]
To show that $(\phi_*,\widetilde{\psi})\neq0$, it suffices to prove that both $\phi_*$ and $\widetilde{\phi}$ are functions of fixed sign. Noting that $L_c\phi_*=\lambda_*\phi_*$ can be rewritten as 
\[
\partial_x^2\rho_*-A(x,\lambda_*)\rho_*=0
\]
where $$A(x,\lambda)\triangleq \frac{c-2k-4\phi(x)-4\lambda}{c-\phi(x)-\lambda}$$
and $\rho_*\triangleq (4-\partial_\xi^2)^{-1}\phi_*$ is even and everywhere positive; see the proof of Theorem \ref{spectrum} in \cite{LLW} for details. 
As a result, we have
\[
\phi= (4-\partial_x^2)\rho_*=(4-A(x,\lambda_*))\rho_*=\frac{3c+2k}{c-\phi-\lambda_*}\rho_*>0.
\]
On the other hand, we recall \[\widetilde{\psi}=(1-\partial_\xi^2)(4-\partial_\xi^2)^{-1}\phi=\phi-3(4-\partial_\xi^2)^{-1}\phi,\] where the profile $\rho\triangleq (4-\partial_\xi^2)^{-1}\phi$ can be expressed in terms of $\phi$. More specifically, the traveling wave equation  \eqref{traveling},
\[
c(1-\partial_x^2)(4-\partial_x^2)^{-1}\phi-[\frac{1}{2}(\phi)^2+(4-\partial_x^2)^{-1}2k\phi]=0,
\]
can be rewritten as
\[
c(\phi-3\rho)=\frac{1}{2}\phi^2 +2k\rho,
\]
which, after simple rearrangements, yields
\begin{equation}\label{e:rho}
    \rho=\frac{2c\phi -\phi^2 }{6c+4k}.
\end{equation}
Thus, plugging \eqref{e:rho} into the expression of $\widetilde{\psi}$, we have
\[
\widetilde{\psi}=\phi-3\rho=\frac{3\phi+4k}{2(3c+2k)}\phi>0,
\]
which concludes the proof that $\widetilde{\lambda}_*>\lambda_*$. Denoting the $L^2$-normalized eigenfunction of $L_c^\Pi$ with respect to $\widetilde{\lambda}_*$ as $\widetilde{\phi}_*$, there exists $b\in\R$ such that 
\[
L_c\widetilde{\phi}_*=\widetilde{\lambda}_*\widetilde{\phi}_*+b\widetilde{\psi},
\]
or, equivalently,
\[
\widetilde{\phi}_*=b(L_c-\widetilde{\lambda}_*)^{-1}\widetilde{\psi}.
\]
Noting that $\widetilde{\phi}_*\in\caa^\prime$ imposes the constraint $(\widetilde{\psi}, \widetilde{\phi}_*)=0$, we introduce the scalar function 
\[
g(\lambda)\triangleq ((L_c-\lambda)^{-1}\widetilde{\psi}, \widetilde{\psi}),
\]
which is real analytic for $\lambda\not\in\sigma(L_c)\backslash\{0\}$ and $g(\lambda)=0$ if and only if $\lambda$ is an eigenvalue of $L_c^\Pi$. Moreover, $g$ is strictly increasing on connected smooth intervals, thanks to the fact that 
\[
g^\prime(\lambda)=\Vert(L_c-\lambda)^{-1}\widetilde{\psi})\Vert^2_{L^2(\R)}\geq 0.
\] More specifically, denoting the smallest element of $\sigma(L_c)\backslash\{0,\lambda_*\}$ as $\lambda_1>0$, we have $g(\lambda)$ is strictly increasing on $(\lambda_*, \lambda_1)$ and admits $\lambda_*$ as a pole. Consequently, 
\[
\widetilde{\lambda}_*>0 \quad \text{if and only if} \quad g(0)<0.
\]
To see that $g(0)<0$, we recall that $\phi$ is a critical point of the Lagrangian $Q_c$, that is,
\[
\frac{\delta Q_c}{\delta u}(\phi)=\frac{\delta H}{\delta u}(\phi)+c\frac{\delta S}{\delta u}(\phi)=0,
\]
which, taken derivative with respect to $c$ to both sides, yields,
\[
L_c\partial_c\phi=-\frac{\delta S}{\delta u}(\phi)=-(1-\partial_x^2)(4-\partial_x^2)^{-1}\phi=-\widetilde{\psi}.
\]
As a result, we have
\[
g(0)=(L_c^{-1}\widetilde{\psi}, \widetilde{\psi})=-(\partial_c\phi, \frac{\delta S}{\delta u}(\phi))=-\frac{d}{d c}S(\phi).
\]
We then conclude from Lemma \ref{l:convex} that $g(0)<0$ and thus $\widetilde{\lambda}_*>0 $.
\end{proof}

\begin{proof}[Proof of Proposition \ref{ineq}] The proof has been given above.
\end{proof}

\bigskip

\noindent{\bf Acknowledgments}  The work of Li is partially supported by the NSFC grant 11771161. The work of Liu is partially supported by the Simons Foundation  grant 499875. The work of Wu is partially supported by the NSF grant DMS-1815079.


\bibliographystyle{plain}

\end{document}